\documentclass[a4paper,11pt]{article}
\usepackage[utf8]{inputenc}
\usepackage[margin=2cm]{geometry}

\usepackage{xcolor}
\definecolor{myblue}{rgb}{0.1 0.1 0.6}
\usepackage{hyperref}
\hypersetup{
   colorlinks=true,
   linkcolor=myblue,
   citecolor=myblue,
   urlcolor=myblue
}

\usepackage[T1]{fontenc}
\usepackage{amsmath,amssymb,amsthm}
\usepackage{newpxtext}
\usepackage{eulerpx}

\usepackage{url}
\usepackage[shortlabels]{enumitem}

\theoremstyle{plain}
\newtheorem{proposition}{Proposition}[section]
\newtheorem{theorem}[proposition]{Theorem}

\newtheorem{lemma}[proposition]{Lemma}

\theoremstyle{remark}
\newtheorem{remarkx}[proposition]{Remark}

\newenvironment{remark}
  {\pushQED{\qed}\remarkx}
  {\popQED\endremarkx}

\newcommand{\mc}[1]{\mathcal{#1}}

\newcommand{\ip}[2]{{\langle {#1} , {#2} \rangle}}
\newcommand{\lin}{\operatorname{lin}}

\newcommand{\proten}{\widehat\otimes}
\newcommand{\Tr}{{\operatorname{Tr}}}
\newcommand{\catHilb}{\textsf{Hilb}}
\newenvironment{smat}{\bigl(\begin{smallmatrix}}{\end{smallmatrix}\bigr)}
\newcommand{\smattwo}[4]{\begin{smat}{#1}&{#2}\\{#3}&{#4}\end{smat}}

\begin{document}

\title{$W^*$-categories are von Neumann}
\author{Matthew Daws}
\maketitle

\begin{abstract}
A $C^*$-category is to a $C^*$-algebra what a groupoid is to a group.  An example is the category of Hilbert spaces and bounded linear maps, and indeed, every $C^*$-category can be represented as a sub-$*$-category of this.  Analogously, a $W^*$-category is a $C^*$-category where every object has a predual, and we might expect a representation theory on Hilbert spaces where the resulting morphism spaces are $\sigma$-weakly closed.  We identify a gap in the literature here, stemming from the (perhaps surprising) fact that a $C^*$-algebra can have a non-isometric predual and yet not be a $W^*$-algebra.  Using the theory of dual TROs (ternary rings of operators) we repair this gap, and along the way, also give further justification to arguments in the literature about dual TROs.  Motivated by the theory of Dual Banach Algebras, we offer an alternative axiomatisation of $W^*$-categories where we allow non-isometric preduals of the hom spaces, but require that they are additionally bimodules in a certain sense.  This involves the theory of self-dual Hilbert $C^*$-modules.

\end{abstract}

\section{Introduction}

A $C^*$-category, see \cite{GLR_Wstar_categories,mitchener_cstar_cats} for example, is a linear category $\mc A$ where each morphism space $\mc A(A,B) = \hom(A,B)$ is a Banach space making composition of morphisms contractive, and there is an involutive anti-linear contravariant endofunctor $*$, denoted $\mc A(A,B) \ni x \mapsto x^*\in\mc A(B,A)$ which satisfies $\|x\|^2 = \|x^*x\|$ and such that $x^*x \in \mc A(A,A)$ is positive (there is $a\in\mc A(A,A)$ with $x^*x = a^*a$).  Notice that $\mc A(A,A)$ is hence a $C^*$-algebra.  An example is given by taking a collection of Hilbert spaces, say $A\in\mc A$ associated to Hilbert space $H_A$, and letting $\hom(A,B) = \mc B(H_A,H_B)$ the set of all bounded linear maps from $H_A$ to $H_B$.  When we consider the class of all Hilbert spaces, this gives the $C^*$-category $\catHilb$.  One can also consider suitable subcategories, and by \cite[Proposition~1.14]{GLR_Wstar_categories} or \cite[Section~6]{mitchener_cstar_cats}, every $C^*$-category is equivalent to a subcategory of $\catHilb$.  See \cite{bunke2025additiveccategoriesktheory} for a recent, thorough, more category-theory oriented approach.

Recall that a $W^*$-algebra is a $C^*$-algebra $A$ which is \emph{isometrically} isomorphic to a dual Banach space, and a von Neumann algebra is a $C^*$-algebra $A \subseteq\mc B(H)$, for some Hilbert space $H$, which is weak$^*$-closed, when $\mc B(H)$ is realised as the dual space of the trace-class operators on $H$.  By von Neumann's bicommutant theorem, $A \subseteq\mc B(H)$ is weak$^*$-closed and unital if and only if $A = A''$.  By definition, a von Neumann algebra is $W^*$, while Sakai's theorem shows the converse: a $W^*$-algebra $A$ is weak$^*$-homeomorphic to a von Neumann algebra, for some choice of Hilbert space $H$.  It is hence common to confuse these two ideas, as they are essentially the same, but we shall find it useful to be a little old-fashioned and to stick to the definitions just given.  See \cite[Section~III.3]{TakesakiI} or \cite{Sakai_Book}, for example.

A $W^*$-category, \cite{GLR_Wstar_categories}, is defined to be a $C^*$-category where every morphism space is (isometrically) isomorphic to a dual Banach space.  For example, $\catHilb$ is $W^*$, as is any subcategory where $\hom(A,B) \subseteq\mc B(H_A,H_B)$ is weak$^*$-closed for the duality with the trace class operators $H_B\to H_A$.  We believe there is a gap in \cite[Lemma~2.6]{GLR_Wstar_categories} (and so many further results from \cite{GLR_Wstar_categories} are in doubt) because there is an explicit use of ``an equivalence of norms'' in the proof.  Unfortunately (and perhaps surprisingly) the theory of $W^*$-algebras really does require \emph{isometric} isomorphism to a dual space.

For example, \cite[Example~6.9.10]{DDLS_CX_spaces_Book} gives a construction of a compact Hausdorff space $X$ such that $C(X)$ is isomorphic to the dual of a Banach space, but $C(X)$ is not a von Neumann algebra, because $X$ is not a hyper-Stonian space.  We can take $X$ to be the \emph{Gleason Cover} of $[0,1]$, which can be constructed as the character space of the commutative $C^*$-algebra given by quotienting $B^b([0,1])$, the bounded Borel functions on $[0,1]$, by the equivalence relation of equality of functions on a meagre set.  As $X$ is Stonian and separable, $C(K)$ is isomorphic to $\ell^\infty$, so actually a bidual, but $X$ has no non-zero normal measures and so is very far from being hyper-Stonian.  The details can be found in \cite{DDLS_CX_spaces_Book}; I first learnt of this example from \cite{MO244675}.

We give a new proof of this problematic lemma.  It is perhaps surprising that this problem has persisted in the literature: we suspect this is because $W^*$-categories which one meets ``in the wild'' are usually explicitly subcategories of $\catHilb$; indeed, the lemma's main use is perhaps to show that every $W^*$-category arises in this way.  Here we see a close parallel with the relation between $W^*$-algebras and von Neumann algebras.  We make extensive use of Ternary Rings of Operators (TROs), see Section~\ref{sec:cats_to_TROs} below, and in particular the very accessible paper \cite{EOR_InjectivityNuclearityOS}.  There appears to be a further gap in a result in \cite{EOR_InjectivityNuclearityOS}, which we also give a new proof of, Proposition~\ref{prop:Vr_vnalg}.  A central idea we use is that of the multiplier algebra of a $C^*$-algebra, and the tight connection with TROs, Section~\ref{sec:mults}.  In Section~\ref{sec:main} we prove our first main result, Theorem~\ref{thm:main}, giving a new proof of \cite[Lemma~2.6]{GLR_Wstar_categories}, that a certain $C^*$-algebra, constructed from two objects in a $W^*$-category, is in fact a $W^*$-algebra.  We then make some comments about continuity of various module actions which naturally arise.

In Section~\ref{sec:self-dual} we briefly remark upon some connections between hom spaces in a $W^*$-category, and self-dual Hilbert $C^*$-modules.  In Section~\ref{sec:dbas} we consider non-isometric preduals which have the additional structure of being bimodules of the dual space, following work of Palmer \cite{palmer_arens_mult_Wstar} and Pham \cite{pham_vn_algs}.
We then consider $C^*$-categories where each hom space has such a ``bimodule predual''.  Rewritten in Hilbert module language, this is very close to \cite[Theorem~2.6]{Schweizer_HilbModsPredual}, but the proof of this result has exactly the same issue of providing a $C^*$-algebra with an isomorphic, but not isometric, predual.  We sketch the setup of \cite{Schweizer_HilbModsPredual}, and then give a new proof of the problematic part of this theorem.  This leads to a new characterisation of $W^*$-categories, Theorem~\ref{thm:wstar_from_bimods}.

In two appendices we give a couple of constructions presumably known to experts, but which we think it is worth giving the details of here, to make this paper more self-contained, and to carefully check certain details.  We mostly introduce necessary terminology and notation as needed.  For a Banach space $E$, we let $E^*$ denote the dual; when $E$ has some form of an adjoint, we denote $E^\sharp = \{ x^* : x\in E \}$, following \cite{EOR_InjectivityNuclearityOS}.  The duality between $E$ and $E^*$ is sometimes denoted $\ip{\mu}{x}$ for $\mu\in E^*, x\in E$.  Our inner-products are linear in the second variable, except where explicitly stated otherwise at points in Section~\ref{sec:dbas}.

\subsection{Acknowledgements}

The author thanks Narutaka Ozawa for useful correspondence about \cite{EOR_InjectivityNuclearityOS}.  I thank Yemon Choi for helpful comments on a preprint of this paper, and thank Alexander Engel for bringing \cite{bunke2025additiveccategoriesktheory} to my attention.

The author wishes to thank the UK's National Health Service (NHS) and the scores of doctors and nurses who cared for him during a very serious illness which occurred in the middle of the writing this paper.

\section{From categories to TROs and modules}\label{sec:cats_to_TROs}

Given a $C^*$-category $\mc A$ we shall write $\hom(A,B)$ for the Banach space of morphisms from $A$ to $B$, given $A,B\in\mc A$.  The $*$ operation gives a bijection $\hom(A,B) \to \hom(B,A)$.  As $\hom(A,A)$ is a $C^*$-algebra, we shall tend to identify $A$ with $\hom(A,A)$ to ease notation.  As we shall mostly be interested in $W^*$-categories, we shall mostly assume that $A$ is unital (see \cite[Section~3]{mitchener_cstar_cats} for why this is not always the correct assumption, if working with $C^*$-categories).

We can view $\hom(A,B)$ as a Hilbert $C^*$-module over $A$, for the $A$-valued inner-product
\[ (x|y) = x^*y \in \hom(A,A) = A \qquad (x,y\in\hom(A,B)). \]
See \cite{Lance_HilbModsBook} for more on Hilbert $C^*$-modules, and also the initial chapters of \cite{RaeburnWilliams_MoritaEquivBook} or \cite[Chapter~8]{BlecherLeMerdy_Book}, for example.  Similarly, $\hom(A,B)$ is a \emph{left} Hilbert $C^*$-module over $B$ for the $B$-valued inner-product $(x|y) = xy^*$.  See later in Section~\ref{sec:dbas} for more on this viewpoint.

Given $A,B\in\mc A$ we can turn the space
\[ M(A,B) = \begin{pmatrix} A & \hom(B,A) \\ \hom(A,B) & B \end{pmatrix}
= \left\{ \begin{pmatrix} a & x \\ y & b \end{pmatrix} : \begin{matrix} a\in A, x\in\hom(B,A) \\
  y \in\hom(A, B), b\in B \end{matrix} 
  \right\} \]  
into a $*$-algebra for the obvious matrix adjoint and multiplication.  Any $*$-functor $\mc A \to \catHilb$ will give a $*$-representation of $M(A,B)$ on a Hilbert space, and so \cite[Proposition~1.14]{GLR_Wstar_categories} or \cite[Section~6]{mitchener_cstar_cats} shows that $M(A,B)$ can be given a $C^*$-norm, which is of course unique once we know existence.  See the end of \cite[Section~1]{GLR_Wstar_categories} for more remarks.

\begin{remark}\label{rem:coms_iso}
Let us introduce the terminology that a \emph{corner} or \emph{component} of $M(A,B)$ is one of the sub-spaces $A,B,\hom(A,B)$ or $\hom(B,A)$.  Suppose $M(A,B)$ has a norm making it a $C^*$-algebra.  Then the canonical map $A \to M(A,B)$ is an injective $*$-homomorphism between $C^*$-algebras, and so an isometry.  Similarly $B\to M(A,B)$.  Given $x\in\hom(A,B)$ we then have
\[ \left\|\smattwo{0}{0}{x}{0}\right\|^2
= \left\| \smattwo{0}{x^*}{0}{0} \smattwo{0}{0}{x}{0} \right\|
= \left\| \smattwo{x^*x}{0}{0}{0} \right\|
= \|x^*x\| = \|x\|^2, \]
using the $C^*$-condition, that $A\to M(A,B)$ is an isometry, and the $C^*$-category axioms.  Thus $\hom(A,B) \to M(A,B)$ is automatically an isometry; similarly $\hom(B,A) \to M(A,B)$.
\end{remark}

\begin{remark}\label{rem:contractive_proj}
The elements $\smattwo{1}{0}{0}{0}$ and $\smattwo{0}{0}{0}{1}$ are projections in $M(A,B)$, and so, for example,
\[ M(A,B) \to M(A,B); \quad \smattwo{a}{x^*}{y}{b} \mapsto 
\smattwo{1}{0}{0}{0} \smattwo{a}{x^*}{y}{b} \smattwo{0}{0}{0}{1}
= \smattwo{0}{x^*}{0}{0} \qquad (a\in A,b\in B,x,y\in\hom(A,B)), \]
gives a contractive projection of $M(A,B)$ onto the image of $\hom(B,A)$.  Similarly the images of $A,B$ and $\hom(A,B)$ are the images of contractive projections.
\end{remark}

Consider a $*$-representation $M(A,B) \to \mc B(H_0)$ for some Hilbert space $H_0$.  By using the projections from Remark~\ref{rem:contractive_proj} (compare the argument in Appendix~\ref{sec:biduals}) one can show that $H_0 \cong H\oplus K$ for some Hilbert spaces $H,K$ and that $M(A,B)$, viewed as a matrix algebra, acts naturally on $H\oplus K$.  That is, $A\subseteq\mc B(H)$ is a non-degenerate $*$-representation, similarly $B\subseteq\mc B(K)$, and $\hom(A,B) \subseteq \mc B(H,K), \hom(B,A) \subseteq \mc B(K,H)$.  In particular, given $x,y,z\in\hom(A,B)$ we have that $y^*\in\hom(B,A)$ so $xy^*\in B$ and $xy^*z \in \hom(A,B)$.  So $\hom(A,B) \subseteq \mc B(H,K)$ is a Ternary Ring of Operators (TRO).

There is an abstract theory of TROs, termed \emph{ternary $C^*$-rings}, studied by Zettl in \cite{Zettl_TROs}.  We believe Zettl's results could be used in the following, but they appear difficult to apply (not least because of the need to rule out the possibility that the operator $T$ from \cite[Theorem~3]{Zettl_TROs} is non-trivial).  Instead, we shall follow the theory for TROs developed in \cite[Section~2]{EOR_InjectivityNuclearityOS}.

\section{Dual TROs}\label{sec:dualTROs}

In this section, we mostly follow the notation of \cite{EOR_InjectivityNuclearityOS}.  So, $V\subseteq \mc B(K,H)$ is a TRO: a norm closed subspace with the property that $xy^*z\in V$ for each $x,y,z\in V$.  Denote
\[ V^\sharp = \{ x^* : x\in V \} \subseteq \mc B(H,K) \]
which is also a TRO.  We write $V^*$ for the Banach space dual of $V$.  The TRO property shows that $V V^\sharp = \lin\{ xy : x\in V, y\in V^\sharp \}$ is a $*$-subalgebra of $\mc B(H)$, so defining $C$ to be the closure, we obtain a $C^*$-algebra.  Similarly $D = \overline{V^\sharp V} \subseteq \mc B(K)$ is a $C^*$-algebra.  These algebras may not be unital, but if we wish, there is no loss of generality in supposing that they are non-degeneate, see Appendix~\ref{sec:make_nondeg}.  As $V$ is norm-closed, we have that $CV \subseteq V, VD\subseteq V$ and so $V$ is a $C$-$D$-bimodule.

When $V$ is weak$^*$-closed in $\mc B(K,H)$ (for the weak$^*$-topology given by identifying $\mc B(K,H)$ with the dual space of the trace-class operators $H\to K$) we obtain a \emph{$W^*$-TRO}.  When $V$ is isometrically isomorphic to a dual Banach space, we say that $V$ is a \emph{dual TRO}.  Write $V_*$ for the Banach space with $V = (V_*)^*$.  Clearly a $W^*$-TRO is a dual TRO.  The main result of \cite[Section~2]{EOR_InjectivityNuclearityOS} is that the converse holds.

Define the ``linking algebra'' of $V$ (in analogy with the construction for Hilbert $C^*$-modules, \cite[Chapter~2]{RaeburnWilliams_MoritaEquivBook}; compare also the definition of $M(A,B)$ from Section~\ref{sec:cats_to_TROs}, and \cite[Section~2]{EOR_InjectivityNuclearityOS}) to be
\[ L(V) = \begin{pmatrix} C & V \\ V^\sharp & D \end{pmatrix}, \]
with the obvious product and $*$-operation.  As $V\subseteq\mc B(K,H)$ we have a natural $*$-representation of $L(V)$ on $H\oplus K$ (which again we may suppose is non-degenerate, by Appendix~\ref{sec:make_nondeg}).  When $V$ is a $W^*$-TRO it is natural to consider $\overline{C}^\sigma$, the weak$^*$-closure of $C$, and similarly for $D$, and so to consider
\[ L_w(V) = \begin{pmatrix} \overline{C}^\sigma & V \\ V^\sharp & \overline{D}^\sigma \end{pmatrix} \subseteq \mc B(H\oplus K). \]
As argued on \cite[pages~493--494]{EOR_InjectivityNuclearityOS}, as $V$ is weak$^*$-closed and as composition of operators is separately weak$^*$-continuous, we have that $\overline{C}^\sigma V \subseteq V$.  Similarly $V \overline{D}^\sigma \subseteq V$, and so $L_w(V)$ is a $*$-algebra.  Notice that $L_w(V)$ is weak$^*$-closed and is the weak$^*$-closure of $L(V)$.  If $V$ is non-degenerate, then $L_w(V)$ is unital, $\overline{C}^\sigma = C''$ and the same for $D$, and so $L_w(V) = L(V)''$ is a von Neumann algebra.

There appears to be a gap in \cite[Proposition~2.4]{EOR_InjectivityNuclearityOS}: in the proof, it is stated, without justification, that certain $C^*$-algebras $V_{(r)}$ are von Neumann algebras.  We now indicate how to show this property, hence giving full justification to the result.  We first continue to introduce some of the concepts involved.

Given a TRO $V$ and a partial isometry $r\in V$ let
\[ V_{(r)} = \{ x\in V : rr^* x = x = x r^*r \} = \{ rr^* x r^*r : x\in V \}, \]
the equality following from the fact that $rr^*, r^*r$ are projections.  This is a unital $C^*$-algebra for the product and involution
\[ x \bullet y = xr^*y, \quad x^\flat = rx^*r \qquad (x,y\in V_{(r)}), \]
with the unit being the element $r$.  The following is mentioned in a number of places, but not proved, so for completeness we give a proof.

\begin{lemma}
The map $\theta \colon V_{(r)} \to rr^* C rr^*; x \mapsto xr^*$ is a $*$-isomorphism with inverse $a\mapsto ar$.  This shows in particular that $V_{(r)}$ is a $C^*$-algebra.
\end{lemma}
\begin{proof}
As $rr^*$ is a projection in $C$, we see that $rr^* C rr^*$ is a $C^*$-algebra, with unit $rr^*$.  For any $x\in V$, indeed $xr^* \in C$, and when $x\in V_{(r)}$ we have that $rr^*(xr^*)rr^* = rr^* x r^* = xr^*$ so $\theta$ does map $V_{(r)}$ to $rr^* C rr^*$.  Also $\theta(x^\flat) = rx^*rr^* = r(rr^*x)^* = rx^* = \theta(x)^*$ and for $x,y\in V_{(r)}$ we have $\theta(x\bullet y) = xr^*yr^* = \theta(x)\theta(y)$, so $\theta$ is a $*$-homomorphism.

Let $a = yz^*$ for some $y,z\in V$ and consider $x = (rr^* a rr^*) r$.  Then $x = rr^* yz^* r = rr^* (yz^* r) \in V$ by the TRO property.  As $V$ is norm closed, it follows that for any $a\in C$ we have that $x = (rr^* a rr^*) r \in V$, and so for any $a\in rr^* C rr^*$ we have that $x = ar \in V$.  For such an $x$ we now see that $rr^* x = rr^* a r = ar$ and $x r^*r = arr^*r = ar = x$ so that $x\in V_{(r)}$.  Finally, $\theta(x) = arr^* = a$ and so we conclude that $\theta$ is a bijection.

As $\|r^*\|\leq 1$ we see that $\theta$ is a contraction, and similarly $\theta^{-1}$ is a contraction, so $\theta$ is an isometry.  Hence the given norm on $V_{(r)}$ satisfies the $C^*$-condition.
\end{proof}

To show that $V_{(r)}$ is a $W^*$-algebra when $V$ is a dual TRO, we shall use the theory of conditional expectations, see \cite[Theorem~2.5]{EOR_InjectivityNuclearityOS}.  When $V \subseteq W$ is a sub-TRO, an idempotent $P \colon W \to W$ with image equal to $V$ is a \emph{conditional expectation} when:
\begin{itemize}
    \item[C1:] $P(xy^*w) = xy^*P(w)$ for $x,y\in V, w\in W$;
    \item[C2:] $P(wx^*y) = P(w) x^*y$ for $x,y\in V, w\in W$;
    \item[C3:] $P(xw^*y) = xP(w)^*y$ for $x,y\in V, w\in W$.
\end{itemize}
We shall not use axiom C3.  A careful examination of the proof of \cite[Theorem~2.5]{EOR_InjectivityNuclearityOS} shows that it does not depend upon \cite[Proposition~2.4]{EOR_InjectivityNuclearityOS}; indeed, we shall only use the claim that when $P$ is contractive, it has properties C1 and C2, and this claim has a fairly simple proof (showing C3 is the harder part).

\begin{proposition}\label{prop:Vr_vnalg}
Let $V$ be a dual TRO, with isometric predual $V_*$.  Regarding $V$ as a $C$-$D$-bimodule, so that $V^*$ is a $D$-$C$-bimodule, we have that $V_* \subseteq V^*$ is a submodule.  Given $r\in V$ a partial isometry, the $C^*$-algebra $V_{(r)}$ is, as a subspace of $V$, weak$^*$-closed, and hence is a $W^*$-algebra.
\end{proposition}
\begin{proof}
The canonical map $V_* \to V_*^{**} = V^*$ has adjoint $P \colon V^{**} \to V$ which is a (contractive) idempotent, once we identify $V$ with a subspace of $V^{**}$ in the canonical way.  We have that $V^{**}$ is a $W^*$-TRO; see Appendix~\ref{sec:biduals} for example.  By \cite[Theorem~2.5]{EOR_InjectivityNuclearityOS}, see the discussion above, $P$ is a conditional expectation, in particular, $P(c\Phi) = c P(\Phi)$ and $P(\Phi d) = P(\Phi) d$ for $\Phi \in V^{**}, c\in C, d \in D$.

Let $\mu\in V_*$ and $d\in D$, so as $V^*$ is a bimodule, $d\cdot\mu\in V^*$.  If $d\cdot\mu \not\in V_*$ then by Hahn-Banach there is $\Phi \in V^{**}$ which annihilates $V_*$ and yet $\Phi(d\cdot\mu) = 1$.  However, then
\[ 1 = \ip{\Phi}{d\cdot \mu} = \ip{\Phi d}{\mu}_{(V^{**}, V^*)} = \ip{P(\Phi d)}{\mu}_{(V,V_*)}
= \ip{P(\Phi)d}{\mu}_{(V,V_*)}, \]
where in brackets we indicate which pairs of Banach spaces are in duality.  That $\Phi$ annihilates $V_*$ means exactly that $P(\Phi) = 0$, but this then gives a contradiction.  Similarly $\mu\cdot c\in V_*$ for each $c\in C$.  We have hence shown that $V_*$ is a submodule.

Let $(x_i)$ be a net in $V_{(r)}$ converging weak$^*$ to $x\in V$.  Thus $x_i = rr^* x_i r^*r$ for each $i$, and so for $\mu\in V_*$, as $r^*r \cdot \mu \cdot rr^* \in V_*$, we see that
\[ \ip{rr^* x r^*r}{\mu} = \ip{x}{r^*r \cdot \mu \cdot rr^*} = \lim_i \ip{x_i}{r^*r \cdot \mu \cdot rr^*}
= \lim_i \ip{rr^* x_i r^*r}{\mu}
= \lim_i \ip{x_i}{\mu} = \ip{x}{\mu}. \]
Hence $rr^* x r^*r = x$ and so $x\in V_{(r)}$.  So $V_{(r)}$ is weak$^*$-closed, and so an (isometric) dual space.
\end{proof}

Let us briefly mention the use of the algebras $V_{(r)}$: that $V$ is dual means that the unit ball of $V$ has many extreme points, and these turn out to be partial isometries, see \cite[Lemma~1.3]{Zettl_TROs}, and this tells us about elements of $V_*$, see \cite[Proposition~4.3]{Zettl_TROs} and \cite[Proposition~2.3]{EOR_InjectivityNuclearityOS}.

However, we shall not use these results directly.  Now we have repaired \cite[Proposition~2.4]{EOR_InjectivityNuclearityOS}, we may use \cite[Theorem~2.6]{EOR_InjectivityNuclearityOS} (see also \cite[Corollary~3.5]{BM_Duality_OpAlgI} for an alternative proof) which in particular tells us that a dual TRO is a $W^*$-TRO for some representation $V \subseteq \mc B(H,K)$, with the weak$^*$-topologies being equal.

\subsection{Multiplier algebras}\label{sec:mults}

Given a $C^*$-algebra $A$ recall the \emph{multiplier algebra} $M(A)$.  This can be realised in a number of different ways:
\begin{itemize}
  \item Given Hilbert $C^*$-modules $E,F$, recall that a linear map $t\colon E\to F$ is adjointable when there is a linear map $t^* \colon F \to E$ with $(x|t(y)) = (t^*(x)|y)$ for each $x\in F, y\in E$.  Then $t$ and $t^*$ are bounded and $(t^*)^*=t$.  We write $\mc L(E,F)$ the space of such maps; $\mc L(E)$ is a $C^*$-algebra.  For $x\in F,y\in E$ we have an operator $\theta_{x,y} \colon E\to F; z\mapsto x \cdot (y|z)$, and then $\mc K(E,F)$, the compact operators, is the closed linear span of such maps.  We have that $\mc K(E)$ is an ideal in $\mc L(E)$, and $M(\mc K(E)) = \mc L(E)$.  See \cite[Chapter~1]{Lance_HilbModsBook} for example.
  
  Regarding $A$ as a Hilbert $C^*$-module over itself for the inner product $(a|b) = a^*b$, we have that $\mc K(A)\cong A$ for the continuous linear extension of the map $\theta_{a,b} \mapsto ab^*$, and so we can realise $M(A)$ as $\mc L(A)$.  See \cite[Chapter~2]{Lance_HilbModsBook} for example.

  \item When $A \subseteq \mc B(H)$ non-degenerately, for some Hilbert space $H$, we have that $M(A)$ is isomorphic to the $C^*$-algebra
  \[ \bigl\{ x\in\mc B(H) : xa, ax\in A \ (a\in A) \bigr\}, \]
  the \emph{centraliser algebra} of $A$.  Then in fact $M(A) \subseteq A'' \subseteq \mc B(H)$.  See \cite[Proposition~3.12.3]{PedersenBook} for example.
\end{itemize}

There is a surprisingly link between $L(V), L_w(V)$ and multiplier algebras, which is outlined in the appendix of \cite{EOR_InjectivityNuclearityOS}.  Again, there is a point we do not follow, and so we give a different proof.  (At the bottom of \cite[page~517]{EOR_InjectivityNuclearityOS}, it is written that ``$MV\subseteq V$ and $VM\subseteq V$'' but we do not see why this is so: rather we think that $V$ is only a left $M$ module, and so we do not follow the subsequent short argument that ``$M=M(C)$''.)

Again let $V\subseteq\mc B(K,H)$ be a TRO, and let $D = \overline{V^\sharp V}$ and $C = \overline{V V^\sharp}$ which are $C^*$-algebras acting on $K$ and $H$ respectively.  Then $V$ is a $C$-$D$-bimodule, and as explained in \cite{EOR_InjectivityNuclearityOS}, it is a faithful bimodule (so if $c\cdot v=0$ for all $v\in V$, then $c=0$, and similarly for the $D$ action.)  An examination of the proof shows that we do not require $V$ to be non-degenerate.

Also $V$ becomes a Hilbert $C^*$-module over $D$ for the inner-product $(u|v) = u^*v$.  Let $\mc K_0(V) = \lin\{\theta_{u,v} : u,v\in V \}$ a dense $*$-subalgebra of the compact operators $\mc K(V)$.  Given $\theta_{u,v}\in\mc K_0(V)$, for $w\in V$ we have $\theta_{u,v}(w) = u\cdot(v|w) = uv^*w = (uv^*)\cdot w$, and so there is a map $\pi \colon \mc K_0(V) \to VV^\sharp \subseteq C; \theta_{u,v} \mapsto uv^*$.  This is well-defined as $V$ is a faithful $C$-module, and also injective (hence bijective) for the same reason; one checks that $\pi$ is a $*$-homomorphism.  For $\theta\in\mc K_0(V), w\in V$ we have that $\theta(w) = \pi(\theta)\cdot w$ and so $\|\theta\|_{\mc K(V)} \leq \|\pi(\theta)\|_C$.  Hence by continuity, $\pi^{-1}$ extends to a $*$-homomorphism $C \to \mc K(V)$, which is again injective as $V$ is faithful.  As $C$ is a $C^*$-algebra, we in fact have an isometric $*$-isomorphism $C \to \mc K(V)$.  In particular, note that the norm on $\mc K(V)$ depends only on $V$, and hence the norm on $VV^\sharp \subseteq C$ (and hence on $C$ itself) depends only on $V$, not the particular embedding $V \subseteq \mc B(K,H)$.

Then $M(C) \cong M(\mc K(V)) \cong \mc L(V)$.  Now suppose that $V \subseteq \mc B(H,K)$ is non-degenerate, so also $C$ acts non-degenerately on $H$.  As above, we have
\[ M(C) \cong \{ x\in C'' : xc, cx\in C \ (c\in C \subseteq C'') \}. \]
Given $x\in M(C) \subseteq C'' \subseteq \mc B(K)$, let $T \in \mc L(V)$ be the related adjointable map, which satisfies that
\[ x (uv^*) w = T(uv^*w) = T \theta_{u,v}(w) = \theta_{T(u), v} w = T(u) v^* w \qquad (u,v,w\in V), \]
and so $xu = T(u)$ for $u\in V$.  That is, the left action of $M(C) \cong \mc L(V)$ on $V \subseteq \mc B(K,H)$ is just composition of operators.

When $V$ is weak$^*$-closed in $\mc B(K,H)$, as above (also \cite[pages~493--494]{EOR_InjectivityNuclearityOS}), letting $C'' = \overline{C}^\sigma$ be the weak$^*$-closure of $C$ in $\mc B(H)$, we have that $V$ is a left $C''$ module.  Given $x\in C''$, as $x V \subseteq V$, left multiplication by $x$ induces a map $T\colon V\to V$; similarly $x^*$ gives a map $S\colon V \to V$.  We then see that
\[ (S(u)|v) = (x^*u|v) = (x^*u)^*v = u^*xv = (u|xv) = (u|T(v)) \qquad (u,v\in V). \]
So $T$ is adjointable with $T^* = S$.  The resulting map $C'' \to \mc L(V)$ is injective, for if $x V = \{0\}$ for some $x\in C''$, then $x VV^\sharp=\{0\}$ so $xC=\{0\}$ and hence $x=0$ as $C$ acts non-degenerately on $H$.  So we have an inclusion $C'' \subseteq \mc L(V)$ and an inclusion $\mc L(V) \subseteq C''$ and the composition is the identity.  We conclude that $\mc L(V) = C''$.

Similarly $V$ becomes a left Hilbert $C^*$-module over $C$, with $D$ the compact operators, and $D''$ the adjointable maps.  We conclude that $C''$ and $D''$ can be constructed just from $V$, independent of the (non-degenerate) embedding $V \to \mc B(K,H)$.

\section{The algebra of a \texorpdfstring{$W^*$}{WStar}-Category is \texorpdfstring{$W^*$}{WStar}}\label{sec:main}

In this section we state and prove our main result.  The technique used follows on from Section~\ref{sec:mults}, making use of multiplier algebras, for which the next lemma is useful.

\begin{lemma}
Let $M$ be a von Neumann algebra and $A \subseteq M$ a $C^*$-subalgebra which is also an ideal.  Denote by $\overline{A}^\sigma$ the weak$^*$-closure of $A$ in $M$.  There is a central projection $z\in M$ such that $\overline{A}^\sigma = zM$, and such that the $*$-homomorphism $M \to M(A)$ given by multiplication (as $A$ is an ideal in $M$) restricts to an isomorphism $zM \cong M(A)$.
\end{lemma}
\begin{proof}
As the product on $M$ is separately weak$^*$-continuous, $\overline{A}^\sigma$ is also an ideal in $M$, and being weak$^*$-closed, there is a central projection $z\in M$ with $zM = \overline{A}^\sigma$, \cite[Proposition~II.3.12]{TakesakiI}.  In particular, for $a\in A$ we have $a=za$ and so $(1-z)a = a(1-z) = 0$.

Let $\pi \colon M \to M(A)$ be the $*$-homomorphism given by multiplication, say $\pi(x)a = xa\in A$ for $x\in M, a\in A$, and so forth.  Then $\pi((1-z)M) = \{0\}$, while if $\pi(x)=0$ then $xa=ax=0$ for each $a\in A$, so also $x\overline{A}^\sigma = \overline{A}^\sigma x = \{0\}$ so $xzM = \{0\}$ so $xz=0$.  Hence $\ker(\pi) = (1-z)M$ and the restriction $\pi \colon zM \to M(A)$ is injective.  Given $\alpha\in M(A)$ and an approximate identity $(e_i)$ in $A$, we have that $(\alpha e_i)$ is a bounded net in $A \subseteq M$, and so, by passing to a subnet if necessary, converges weak$^*$ to $x\in \overline{A}^\sigma = zM$ say.  For $a\in A$ we have $xa = \lim_i \alpha e_i a = \alpha a$, and then for $a,b\in A$ we have $(ax)b = a(xb) = a(\alpha b) = (a\alpha)b$, so also $ax = a\alpha$, that is, $\pi(x) = \alpha$.  So $M(A) \cong zM$ as claimed.
\end{proof}

\begin{theorem}\label{thm:main}
Let $\mc A$ be a $W^*$-category, and let $A,B \in \mc A$.  The algebra $M(A,B)$ is a $W^*$-algebra, and the isometric inclusions $A, B \to M(A,B)$ and $\hom(A,B), \hom(B,A) \to M(A,B)$ are weak$^*$-weak$^*$-homeomorphisms onto their ranges.
\end{theorem}
\begin{proof}
Let $V = \hom(B,A)$, which by Section~\ref{sec:cats_to_TROs} can be realised isometrically as a TRO.  As $V$ is a dual TRO, by \cite[Theorem~2.6]{EOR_InjectivityNuclearityOS} we can find an embedding $V \subseteq \mc B(K,H)$ such that $V$ becomes a $W^*$-TRO, the weak$^*$-topology being that given by $\hom(B,A)_*$.  By Appendix~\ref{sec:make_nondeg} we may suppose that $V$ acts non-degenerately.  Again set $C = \overline{V V^\sharp}, D=\overline{V^\sharp V}$ so we have the linking algebra
\[ L(V) = \begin{pmatrix} C & V \\ V^\sharp & D \end{pmatrix} \subseteq \mc B(H\oplus K). \]
By Section~\ref{sec:mults}, we know that $C'' = \overline{C}^\sigma = M(C)$ and similarly for $D$, and so
\[ L_w(V) = L(V)'' = \begin{pmatrix} \overline{C}^\sigma & V \\ V^\sharp & \overline{D}^\sigma \end{pmatrix}
= \begin{pmatrix} M(C) & V \\ V^\sharp & M(D) \end{pmatrix}. \]

Alternatively, as in Section~\ref{sec:cats_to_TROs}, we also have the $C^*$-algebra
\[ M(A,B) = \begin{pmatrix} A & \hom(B,A) \\ \hom(A,B) & B \end{pmatrix}
= \begin{pmatrix} A & V \\ V^\sharp & B \end{pmatrix}. \]
Furthermore, $A$ is a $W^*$-algebra, and $VV^\sharp = \hom(B,A)\circ\hom(A,B) \subseteq A$ so we can regard $C$ as a sub-$C^*$-algebra of $A$; here recall the discussion in Section~\ref{sec:mults} showing that $C$ and $D$ depend only on $V$, not the (possibly degenerate) realisation of $V$ as a space of operators between Hilbert spaces.

Furthermore, $C$ is an ideal in $A$, as for $x\in A, u,v\in V$ we have $x(uv^*) = (xu)v^* \in VV^\sharp \subseteq C$, because $xu \in V$.  So by density, $AC \subseteq C$, and similarly on the other side.  Apply the lemma to the inclusion $C \subseteq A$ to find a central projection $z\in A$ with $\overline{C}^\sigma = zM \cong M(C)$.  Apply exactly the same argument to $B$ and $D$ to find a central projection $z'$ in $B$ with $M(D) \cong z'B$.  We then have an isomorphism of $*$-algebras
\begin{align*}
M(A,B)
&\cong \begin{pmatrix} zA \oplus (1-z)A & V \\ V^\sharp & z'B \oplus (1-z')B \end{pmatrix}
\cong \begin{pmatrix} zA & V \\ V^\sharp & z'B \end{pmatrix}
   \oplus \begin{pmatrix} (1-z)A & 0 \\ 0 & (1-z')B \end{pmatrix} \\
&\cong \begin{pmatrix} M(C) & V \\ V^\sharp & M(D) \end{pmatrix}
   \oplus \begin{pmatrix} (1-z)A & 0 \\ 0 & (1-z')B \end{pmatrix}
\cong L_w(V) \oplus \begin{pmatrix} (1-z)A & 0 \\ 0 & (1-z')B \end{pmatrix}.
\end{align*}
The direct sum at the end is the direct sum of two $W^*$-algebras, and hence if given the $\max$ norm is itself a $W^*$-algebra.  So $M(A,B)$, a $C^*$-algebra, is $*$-isomorphic to a $W^*$-algebra, and is hence a $W^*$-algebra itself.

The map $A \to M(A,B)$ factors through the map $A \to zA \oplus (1-z)A$ which is a weak$^*$-continuous isomorphism, and similarly for $B\to M(A,B)$.  By construction, the map $V \to M(A,B)$ factors through $V\to L_w(V)$ which is weak$^*$-weak$^*$-continuous.

Now swap the roles of $A$ and $B$, and consider $M(B,A)$ and, say, $U = \hom(A,B) = V^\sharp$.  The same reasoning shows that $U \to M(B,A)$ is weak$^*$-weak$^*$-continuous.  The map
\[ M(A,B) \to M(B,A); \quad \begin{pmatrix} a & u \\ v^* & b \end{pmatrix}
\mapsto \begin{pmatrix} b & v^* \\ u & a \end{pmatrix}, \]
is a $*$-homomorphism and a weak$^*$-weak$^*$-homeomorphism (given the (unique) $W^*$-structure we constructed on $M(A,B)$).  Hence we see that also $\hom(A,B) \to M(A,B)$ is weak$^*$-weak$^*$-continuous.  The proof is now complete following an application of Lemma~\ref{lem:iso_plus_ws_means_wsh} below (or one could use that \cite[Theorem~2.6]{EOR_InjectivityNuclearityOS} tells us that $V$ (and $V^\sharp$) has a unique predual).
\end{proof}


\begin{lemma}\label{lem:iso_plus_ws_means_wsh}
Let $E,F$ be Banach spaces, and let $T \colon F^* \to E^*$ be weak$^*$-continuous and bounded below.  Then $T$ is a weak$^*$-weak$^*$-homeomorphism onto its range.  
\end{lemma}
\begin{proof}
This is standard Banach space theory, though we do not know a precise reference, so we give a sketch proof, using results which, for example, can be found in \cite[Section~3.1]{MegginsonBook}.  As $T$ is weak$^*$-continuous, there is $T_* \in \mc B(E,F)$ with $T = (T_*)^*$, see \cite[Theorem~3.1.11]{MegginsonBook}.  As $T$ is bounded below, it has norm-closed range, and so also has weak$^*$-closed range, \cite[Theorem~3.1.21]{MegginsonBook}.  Setting $E_0 = {}^\circ T(F^*) = \{ x\in E : \ip{T(\mu)}{x}=0 \ (\mu\in F^*) \}$ we have that $T(F^*) = \overline{T(F^*)}^{w^*} = ({}^\circ T(F^*))^\circ = E_0^\circ = \{ \lambda\in E^* : \ip{\lambda}{x}=0 \ (x\in E_0) \}$.  By Hahn--Banach, we identify $(E/E_0)^*$ with $E_0^\circ$ isometrically.  The $\sigma(E_0^\circ, E/E_0)$ topology agrees with the restriction of the $\sigma(E^*,E)$ topology to $E_0^\circ$.  Let $S \colon F^* \to (E/E_0)^*$ be the corestriction of $T$, so $S$ is weak$^*$-continuous, and so again there is $S_* \colon E/E_0 \to F$ with $S = (S_*)^*$.  As $S$ is bounded below and surjective, it is an isomorphism, and so also $S_*$ is an isomorphism, \cite[Theorem~3.1.18]{MegginsonBook}.  Hence $S$ is a weak$^*$-weak$^*$-homeomorphism, and so also $T$ is a weak$^*$-weak$^*$-homeomorphism onto its range.
\end{proof}

The module structure of the hom spaces extends to (pre)duals.  For example, given $\mu\in B^*$ and $x\in\hom(B,A)$ by $\mu\cdot x$ we mean the (bounded) functional $\hom(A,B)\to\mathbb C; y\mapsto \mu(yx)$.  Given $\mu\in\hom(A,B)^*$ let $\mu^*$ be the functional $\hom(B,A) \to \mathbb C; x\mapsto \overline{\mu(x^*)}$, so $\mu^* \in \hom(B,A)^*$ and $(\mu^*)^* = \mu$.

The following proposition is a little strange, because we now know that condition \ref{prop:via_bimod:one} holds always!  It tells us then that conditions \ref{prop:via_bimod:two} and \ref{prop:via_bimod:three} always hold, so there is a tight connection between the preduals of the various hom spaces.  That \ref{prop:via_bimod:two}$\implies$\ref{prop:via_bimod:one} and/or \ref{prop:via_bimod:three}$\implies$\ref{prop:via_bimod:one} shows a much more elementary way to prove Theorem~\ref{thm:main}, \emph{supposing} we had some independent to show that \ref{prop:via_bimod:two} or \ref{prop:via_bimod:three} holds: but this seems out of reach!  The following proof also gives an alternative way to construct a normal functor $\mc A \to \catHilb$, by-passing the topological considerations of \cite[Proposition~2.13]{GLR_Wstar_categories}.

\begin{proposition}\label{prop:via_bimod}
Let $\mc A$ be a $W^*$-category.  The following are equivalent:
\begin{enumerate}[(a)]
  \item\label{prop:via_bimod:one}
  $M(A,B)$ is a $W^*$-algebra;
  \item\label{prop:via_bimod:two}
  for each $A,B\in\mc A, \omega\in A_*$ and $x\in\hom(B,A)$ we have that $\omega\cdot x \in \hom(A,B)_*$, and for $\omega\in\hom(A,B)_*$ we have that $\omega^* \in \hom(B,A)_*$;
  \item\label{prop:via_bimod:three}
  for each $A,B\in\mc A$, with $A_1, A_2 \in \{A,B\}$, with $x\in \hom(A, A_1), y\in\hom(A_2, A), \omega \in A_*$ we have that $x\cdot\omega\cdot y \in \hom(A_1, A_2)_*$.
\end{enumerate}
\end{proposition}
\begin{proof}
Suppose \ref{prop:via_bimod:one} holds.  As $M(A,B)$ is a $W^*$-algebra, the structure of the projection considered in Remark~\ref{rem:contractive_proj}, of $M(A,B)$ onto $A$, is weak$^*$-weak$^*$-continuous, and so $A_*$ can be identified with a contractively complemented subspace of $M(A,B)_*$.  Thus, for $\omega\in A_*$ there is $\omega' \in M(A,B)_*$ with $\omega(a) = \omega'(\smattwo{a}{0}{0}{0})$ for each $a\in A$.  Then, for $x\in \hom(B,A)$ we see that
\[ \omega\cdot x \colon \hom(A,B) \ni y \mapsto \omega(xy)
= \omega'\bigl( \smattwo{0}{x}{0}{0} \smattwo{0}{0}{y}{0}  \bigr)
= \mu \bigl(\smattwo{0}{0}{y}{0} \bigr), \]
say, where $\mu\in M(A,B)_*$.  As the inclusion $\hom(A,B) \to M(A,B)$ is weak$^*$-continuous it follows that $\omega\cdot x$ is as well, as required.  Similarly $\omega\mapsto\omega^*$ is weak$^*$-continuous, showing \ref{prop:via_bimod:two}.  That \ref{prop:via_bimod:one} implies \ref{prop:via_bimod:three} is similar.

To show the converse, our aim is to perform a GNS construction, along the lines of \cite[Proposition~1.8]{GLR_Wstar_categories}.  We follow the detailed construction in \cite[Section~6]{mitchener_cstar_cats}.  Let $A\in\mc A$ and $\omega\in A^*$ be a state.  Given $B\in\mc A$ we let $H_B$ be the separation completion of $\hom(A,B)$ for the pre-inner-product given by $\ip{x}{y} = \omega(x^*y)$.  Denote by $[x]$ the image of $x\in\hom(A,B)$ in $H_B$.  Then for $C\in\mc A$ and $y\in\hom(B,C)$ we have that $\| [yx] \|_{H_C} \leq \|y\|_C \|[x]\|_{H_B}$.  Recalling that $\catHilb$ is the $W^*$-category of Hilbert spaces and bounded linear maps, we obtain a $*$-functor $F' \colon \mc A \to \catHilb$ with $F'(B) = H_B$ and $F'(y)$ the continuous extension of the map $[x] \mapsto [yx]$.

We claim that if $\omega$ is normal on $A$, then $F'$ is weak$^*$-continuous on hom spaces.  We wish to show that $F' \colon \hom(B,C) \to \hom(F(B),F(C)) = \mc B(H_B, H_C)$ is weak$^*$-continuous.  To do this, it suffices to show that if $(x_i)$ is a \emph{bounded} net in $\hom(B,C)$ converging weak$^*$ to $0$, then $F'(x_i) \to 0$ weak$^*$ in $\mc B(H_B, H_C)$.  We are not aware of a canonical reference for this claim, but it's shown by elementary means in \cite[Section~10]{Daws_multipliers}, for example.  As $(F'(x_i))$ is a bounded net, it suffices to check that $(\xi|F'(x_i)\eta) \to 0$ for each $\xi\in H_C, \eta\in H_B$.  As $\{ [x] : x\in\hom(A,B)\}$ is dense in $H_B$, and so forth, it further suffices to check that $([y]|F'(x_i)[x]) \to 0$ for each $x\in\hom(A,B), y\in\hom(A,C)$.  However, $([y]|F'(x_i)[x]) = \omega(y^* x_i x) = (x\cdot\omega\cdot y^*)(x_i)$ and so it suffices (and is necessary) to show that $x\cdot\omega\cdot y^* \in \hom(B,C)_*$.

When \ref{prop:via_bimod:two} holds, we have that $\mu = \omega\cdot y^* \in \hom(A,C)_*$.  Further, $\mu^* \in \hom(C,A)_*$ and so $\mu^* \cdot x^* \in \hom(C,B)_*$, but then for $z\in\hom(B,C)$ we have
\[ (\mu^*\cdot x^*)^*(z)
= \overline{ (\mu^*\cdot x^*)(z^*) }
= \overline{ \mu^*(x^*z^*) }
= \mu(zx)
= \omega(y^*zx)
= (x\cdot\omega\cdot y^*)(z), \]
and so $x\cdot\omega\cdot y^* = (\mu^*\cdot x^*)^* \in \hom(B,C)_*$ as required.  So $F'$ is weak$^*$-continuous on hom spaces.  We now take the direct sum over sufficiently many normal states $\omega$ to ensure that $F' \colon A \to \mc B(H_A)$ is faithful.  Then $F'$ induces a $*$-homomorphism $M(A,B) \to \mc B(H_A \oplus H_B)$ which is faithful (hence isometric) on the $A$ component, and which is weak$^*$-continuous on every component.  Swapping the roles of $A$ and $B$ we obtain a $*$-representation of $M(A,B)$ which is faithful on the $B$ component, and weak$^*$-continuous on each component.  Finally, taking the direct sum of these representations gives a $*$-representation, $\pi$ say, of $M(A,B)$  which is faithful (hence isometric) on both $A$ and $B$, and hence also isometric on the $\hom(A,B)$ and $\hom(B,A)$ components, compare Remark~\ref{rem:coms_iso}.  As $M(A,B)$ is \emph{isomorphic} to a dual space (the direct sum of the preduals of the components) by Lemma~\ref{lem:iso_plus_ws_means_wsh} we conclude that $\pi$ has weak$^*$-closed range.  So the image of $\pi$ is a von Neumann algebra, but $\pi$ is faithful, so $M(A,B)$ is a $W^*$-algebra.

When alternatively \ref{prop:via_bimod:three} holds, we note that we do not need to define $F'$ on all of $\mc A$, but only on the full subcategory generated by $A,B$, and then \ref{prop:via_bimod:three} immediately gives the required condition.
\end{proof}

\begin{remark}\label{rem:alt_bimod_conditions}
An alternative way to phrase the first part of condition \ref{prop:via_bimod:two} is that the composition map $\hom(B,A) \times \hom(A,B) \to \hom(A,A)=A$ is weak$^*$-continuous in the 2nd variable.  Similarly, \ref{prop:via_bimod:three} is equivalent to this composition map being weak$^*$-continuous in both variables, and additionally that $\hom(B,A) \circ B \circ \hom(A,B) \to A$ is weak$^*$-continuous in the middle variable.

In Section~\ref{sec:dbas} we give an alternative ``module'' structure, though the proof is not self-contained in the same way the previous result is.
\end{remark}

\section{Links with self-dual modules}\label{sec:self-dual}

We return to Section~\ref{sec:cats_to_TROs}, and think of $\hom(A,B)$ as a Hilbert $C^*$-module over $A$.  As remarked at the end of Section~2 of \cite{GLR_Wstar_categories} there is a link with ``self-dual modules'' here.  A Hilbert $C^*$-module $E$ over a $C^*$-algebra $A$ is \emph{self-dual} if every bounded module map $t \colon E \to A$ (so $t(x\cdot a) = t(x)a$ for $x\in E, a\in A$) is of the form $t(x) = (t_0|x)$ for some $t_0\in E$; see \cite{paschke_inner_prod_mods} for example.  It is shown in \cite[Proposition~3.8]{paschke_inner_prod_mods} that when $E$ is a self-dual module over a $W^*$-algebra we have that $E$ is isometrically a dual space.  Furthermore, then $\mc L(E)$ is a $W^*$-algebra, \cite[Proposition~3.10]{paschke_inner_prod_mods}, and $E$ admits a polar-decomposition theorem, \cite[Proposition~3.11]{paschke_inner_prod_mods}.  We see the analogy with the spaces $\hom(A,B)$; indeed, \cite[Proposition~2.14]{GLR_Wstar_categories} shows (in particular) that when $\mc A$ is a $W^*$-category, $\hom(A,B)$ is a self-dual Hilbert $C^*$-module over $A$.

\begin{proposition}\label{prop:mod_is_hom_space}
Let $A$ be a $C^*$-algebra and let $E$ be a Hilbert $C^*$-module over $A$.  Then there is a $C^*$-category $\mc A$ with $A$ as an object, and with $B\in\mc A$ with $E = \hom(A,B)$ as a Hilbert $C^*$-module.  When $A$ is a $W^*$-algebra and $E$ is self-dual, we may take $\mc A$ to be a $W^*$-category.
\end{proposition}
\begin{proof}
Simply set $B = \mc K(E)$ the compact operators, let $\mc A$ have two objects, $A$ and $B$, set $\hom(A,B) = E$ and $\hom(B,A) = \overline{E}$ the conjugate space to $E$, say $\overline E = \{ x^* : x\in E \}$.  Then $\overline E$ is a Hilbert $C^*$-module over $B$ via $(x^*|y^*) = \theta_{x,y} \in \mc K(E)$ (compare here \cite[Example~2.11]{RaeburnWilliams_MoritaEquivBook}).  We remark that it is not obvious to us that $\theta_{x,x}$ is positive in the $C^*$-algebra $\mc K(E)$: one way to see this is to use \cite[Lemma~4.1]{Lance_HilbModsBook} (a result attributed to Paschke) that $t\in\mc L(E)$ is positive if and only if $(y|t(y)) \geq 0$ for $y\in E$, and then to note that $(y|\theta_{x,x}(y)) = (x|y)^*(x|y) \geq 0$ for each $y$.  It is more routine to show that the operator norm of $\theta_{x,x}$ agrees with $\|x\| = \|(x|x)\|^{1/2}$, see \cite[Proposition~1.7(ii)]{BMS_quasimults} for example.

Clearly $\hom(A,B)$ is a $B$-$A$-bimodule.  We turn $\hom(B,A)$ into an $A$-$B$-module by $a\cdot x^* = (x\cdot a^*)^*$ and $x^* \cdot \theta = (\theta^*(x))^*$ for $a\in A, x\in E, \theta\in\mc K(E)$.  This is a bimodule structure as $(a\cdot x^*)\cdot \theta = \theta^*(x\cdot a^*)^* = (\theta^*(x)\cdot a^*)^* = a\cdot (x^*\cdot\theta)$ as $\theta^*$ is an $A$-module map.  Then $* \colon \hom(A,B) \to \hom(B,A)$ behaves as required: the product satisfies $(xa)^* = a^*x$ and so forth.

We next check that $\hom(A,B) \times \hom(B,A) \to B$ is a $B$-bimodule map.  This map is $(x,y^*) \mapsto \theta_{x,y}$ which is a $B=\mc K(E)$-module map.  It is also compatible with the $A$ action, for the resulting product satisfies $(xa)y^* = \theta_{xa,y} = \theta_{x,y a^*} = x(ay^*)$.  Similarly $\hom(B,A) \times \hom(A,B) \to A$ is $(x^*,y) \mapsto (x|y)$ is an $A$-bimodule map, and is compatible with the $B=\mc K(E)$ action because $(\theta(x)|y) = (x|\theta^*(y))$ for each $\theta\in B$.  We have hence defined an associative product on all the hom spaces.  The remaining axioms for a $C^*$-category are easy to check.

To finish to proof, we first note that it is not necessary to take $B = \mc K(E)$: any $C^*$-algebra $B$ with $\mc K(E) \subseteq B \subseteq \mc L(E)$ would work equally well.  Now suppose that $A$ is a $W^*$-algebra and $E$ is self-dual.  Then $E$ is isometrically a dual space, as is $\overline E$, and $\mc L(E)$ is a $W^*$-algebra.  So we simply take $B = \mc L(E)$ and then every hom space in $\mc A$ has a predual, so $\mc A$ is a $W^*$-category.
\end{proof}

There is also a link with isometric preduals here, made explicit by \cite[Corollary~3.5]{BM_Duality_OpAlgI}.  Indeed, let $A$ be a $C^*$-algebra, $E$ a Hilbert $C^*$-module over $A$, and suppose that $E$ is isometrically a dual space.  Let $C = \overline\lin\{ (x|y) : x,y\in E\}$ so $C$ is an ideal in $A$, and $E$ is a full module over $C$.  Then \cite[Corollary~3.5]{BM_Duality_OpAlgI} shows (amongst other things) that $M(C)$ is a $W^*$-algebra, that $E$ is self-dual over $M(C)$, and that the predual of $E$ is unique.  We remark that, using that $C$ is an ideal in both $A$ and $M(C)$, one can show that $E$ is also self-dual over $C$, and over $A$, compare \cite[Lemma 8.5.2]{BlecherLeMerdy_Book}.

\section{DBAs, \texorpdfstring{$C^*$}{Cstar}-algebras and non-isometric preduals}\label{sec:dbas}

In this section, we look at weakening the ``isometric'' condition, by considering preduals which are merely isomorphic, but with an additional algebraic property.  To be precise, a \emph{dual Banach algebra} is a Banach algebra $A$, together with a Banach space $E$ and an isomorphism (maybe not isometric) $\phi \colon A \to E^*$ such that, if we use $\phi$ to give $E^*$ a bounded algebra product, then the product is separately weak$^*$-continuous.  By \cite[Proposition~2.1]{daws_dba} we can, and shall, renorm $E$ so that $B = E^*$ becomes a Banach algebra (that is, the product is contractive), and as observed in \cite[Lemma~2.2]{daws_dba}, that the product is separately weak$^*$-continuous is equivalent to $\kappa_E(E) \subseteq E^{**} = B^*$ being a $B$-submodule, where here $\kappa_E \colon E \to E^{**}$ is the canonical map from a Banach space to its bidual.  Equivalently, $\phi^*\kappa_E(E) \subseteq A^*$ is an $A$-bimodule.  We use the terminology of \cite{Runde_DBA}, but the idea is older, see for example \cite{palmer_arens_mult_Wstar}.

The following is not immediate, given that it is false without the additional property that $E_*$ is a submodule (see the counter-example explained in the introduction!)  This result improves Palmer's old result \cite[Theorem~2]{palmer_arens_mult_Wstar} in that we make no assumption about how the adjoint on $A$ and the predual $E_*$ interact.  It was first shown by  Pham who gave a more topological proof; for this reason, and to be self-contained, we give a simpler ``projection'' proof (which itself can be compared to the proof of \cite[Proposition~2.8]{pham_vn_algs}).

\begin{theorem}[{\cite[Corollary~2.6]{pham_vn_algs}}]\label{thm:cstar_dba}
Let $A$ be a $C^*$-algebra which is also a dual Banach algebra, say with predual $E_*$.  Then $A$ is a $W^*$-algebra and $E_*$ is isomorphic to the canonical (isometric) predual $A_*$.
\end{theorem}
\begin{proof}
For any Banach algebra $B$, there are two algebra products, the \emph{Arens products}, see \cite[Section~1.4]{Palmer_Book1} for example, which can be constructed on $B^{**}$ such that $\kappa_B \colon B \to B^{**}$ becomes an algebra homomorphism: we either extend the product by weak$^*$-continuity on one side, or the other, of the product.  An algebra is \emph{Arens regular} when these products coincide, and it is known that $C^*$-algebras are Arens regular (this follows from a calculation using e.g. the universal representation of $A$, for example \cite[Lemma~3]{palmer_arens_mult_Wstar}).  Now set $B = (E_*)^*$ and denote by $\phi \colon A\to B$ the Banach algebra isomorphism.  Then $B$ is Arens regular, and $\phi^{**} \colon A^{**} \to B^{**}$ is an algebra isomorphism.

For any Banach space $E$ we have that $\kappa_E^* \circ \kappa_{E^*} \colon E^* \to E^*$ is the identity, and so $\kappa_E^* \colon E^{***} \to E^*$ can be regarded as a projection, often called the \emph{Dixmier projection}, see \cite[Proposition~1.2]{Runde_DBA} for example.  In our case, we obtain $\kappa_E^* \colon B^{**} \to B$, and this can be verified to be an algebra homomorphism (in general, for either Arens product).  So $\ker(\kappa_E^*) \subseteq B^{**}$ is an ideal, and $\sigma(B^{**}, B^*)$-closed (as the kernel of the Banach space adjoint of an operator).  So $I = (\phi^{-1})^{**}(\ker(\kappa_E^*)) \subseteq A^{**}$ is an ideal and is $\sigma(A^{**}, A^*)$-closed.  As $A^{**}$ is a $W^*$-algebra, there is a central projection $z\in A^{**}$ with $I = (1-z)A^{**}$.  Furthermore, $z' = \phi^{**}(z) \in B^{**}$ is a central idempotent, and $\ker(\kappa_E^*) = (1-z') B^{**}$.

We claim that $T \colon A \to zA^{**}; a\mapsto z\kappa_A(a)$ is a $*$-isomorphism.  As $z$ is central and $\kappa_A$ is a $*$-homomorphism, $T$ is a $*$-homomorphism.  If $T(a)=0$ then $\kappa_A(a) = (1-z)\kappa_A(a)$ and so $\phi^{**}\kappa_A(a) \in \ker(\kappa_E^*)$.  However, $\phi^{**}\kappa_A(a) = \kappa_B\phi(a)$ so $0 = \kappa_E^* \kappa_B\phi(a) = \phi(a)$ so $a=0$ and we conclude that $T$ is injective.  Let $x\in A^{**}$ and set $b = \kappa_E^*\phi^{**}(x) \in B$.  Then
\[ T\phi^{-1}(b) = z \cdot \kappa_A\phi^{-1}(b)
= z\cdot (\phi^{**})^{-1}\kappa_B \kappa_E^*\phi^{**}(x)
= (\phi^{**})^{-1}\bigl( z'\cdot \kappa_B \kappa_E^*\phi^{**}(x) \bigr). \]
(We write $\cdot$ to denote the product here to avoid confusion with the composition of linear maps.)
As $\kappa_B \kappa_E^*$ is an idempotent, $\phi^{**}(x) = \kappa_B \kappa_E^*\phi^{**}(x) + x'$ for some $x' \in \ker\kappa_E^* = (1-z')B^{**}$, so that $z'x'=0$.  Hence $T\phi^{-1}(b) = (\phi^{**})^{-1}(z' \phi^{**}(x)) = z x$, and we conclude that $T$ is surjective.

So $A$ is $*$-isomorphic to $zA^{**}$ which is a W$^*$-algebra, and so $A$ is a $W^*$-algebra.  That $A_* = E_*$ now follows from the uniqueness result \cite[Theorem~5.2]{DHW_conditions_wstar_top} that a $W^*$-algebra has a unique dual Banach algebra predual.
\end{proof}

Our aim now is to show that when we have a $C^*$-category with the property that each $\hom(A,B)$ has an isomorphic predual $\hom(A,B)_*$ which is an $A$-$B$-module, then we have a $W^*$-category.  Note that this condition is not that suggested by Proposition~\ref{prop:via_bimod}, but is somewhat more ``local'' in character.  As in Section~\ref{sec:self-dual}, there are links here with Hilbert $C^*$-modules, and indeed, \cite[Theorem~2.6]{Schweizer_HilbModsPredual} is extremely relevant.  Unfortunately, the proof of this result suffers from exactly the same problem identified in \cite{GLR_Wstar_categories}: a $C^*$-algebra is considered and an \emph{isomorphic} predual constructed.  This does not give a $W^*$-algebra, in general, and so we need to find an alternative argument.  We first indicate how the first part of the proof of \cite[Theorem~2.6]{Schweizer_HilbModsPredual} works, and then present an alternative argument to finish the proof.

In the following, $\proten$ denotes the projective tensor product of Banach spaces.  Given Banach spaces $E,F$ the dual space of $E\proten F$ is identified with $\mc B(F,E^*)$ for the duality $\ip{T}{x\otimes y} = \ip{T(y)}{x}$ for $x\otimes y\in E\otimes F \subseteq E\proten F$ and $T\in\mc B(F,E^*)$.  For a general reference see \cite{Ryan_TensorBook} for example.

\begin{proposition}[{\cite[Theorem~2.6]{Schweizer_HilbModsPredual}}]\label{prop:KE_wstar_selfdual}
Let $A$ be a $C^*$-algebra and $E$ a Hilbert $C^*$-module over $A$.  Consider the action of $\mc K(E)$ on (the left of) $E$, and suppose $E$ has a (possibly not isometric) predual such that for each $\theta\in\mc K(E)$ the map $E\to E; x\mapsto \theta(x)$ is weak$^*$-continuous.  Then $E$ is self-dual.
\end{proposition}
\begin{proof}
We follow the idea of the first part of the proof of \cite[Theorem~2.6]{Schweizer_HilbModsPredual}, but we give some extra details, as we find \cite{Schweizer_HilbModsPredual} hard to follow.  Set $B = \mc K(E)$.  The weak$^*$-continuity assumption on the predual $E_*$ is equivalent to $E_*$ being a right $B$-module, in a way compatible with the left $B$-module structure of $E$.  Define
\[ E_* \proten_B E = E_* \proten E \ / \ \overline\lin\{ \mu\otimes\theta \cdot x - \mu\cdot\theta \otimes x : x\in E, \mu\in E_*, \theta\in B \}. \]
As $E_* \proten_B E$ is a quotient of $E_*\proten E$, we see that $(E_*\proten_B E)^*$ is a subspace of $\mc B(E,E)$, namely $(E_*\proten_B E)^* = \{ t\in\mc B(E) : t(\theta\cdot x) = \theta\cdot t(x) \ (x\in E, \theta\in B) \}$ the algebra of $B$-module maps on $E$, denoted $\mc B_B(E)$.  Note that as $E_*$ is not assumed to be an isometric predual, the dual space $(E_*\proten_B E)^*$ is also not necessarily isometric to $\mc B_B(E)$, only isomorphic.

Again denote by $C$ the closed linear span of $\{ (x|y) : x,y\in E \}$, so $C$ is an ideal in $A$, and $E$ can be regarded as a full module over $C$.  Let $T \colon E \to C$ be a bounded $C$-module map.  As $\mc K(E)$ has an approximate identity, by density, we can find a net $(e_i)$ in $\lin\{ \theta_{x,y} : x,y\in E \} \subseteq \mc K(E)$ such that $\|e_i\|\leq 1$ for each $i$, and with $\lim_i e_i \theta = \theta = \lim_i \theta e_i$ in norm, for each $\theta\in\mc K(E)$.  For each $i$, let $e_i = \sum_{j=1}^n \theta_{\xi_j, \eta_j}$ say, and set $x_i = \sum_{j=1}^n \eta_j \cdot T(\xi_j)^*$, so that
\[ T(e_i \cdot x) = \sum_{j=1}^n T(\xi_j \cdot (\eta_j|x))
= \sum_{j=1}^n T(\xi_j) \cdot (\eta_j|x)
= \sum_{j=1}^n (\eta_j \cdot T(\xi_j)^*|x)
= (x_i|x)  \qquad (x\in E). \]
Then $\|x_i\|^2 = \| (x_i|x_i) \| = \| T(e_i\cdot x_i) \| \leq \|T\| \|x_i\|$, and so $(x_i)$ is a bounded net in $E$.  By moving to a subnet if necessary, we may suppose that $x_i \to x_0$ weak$^*$, for some $x_0\in E$.

Each $a\in C$ gives rise to $t_a \in \mc B_B(E)$ defined by $t_a(x) = x\cdot a$ for $x\in E$; indeed, for $y,z\in E$ we have $t_a(\theta_{y,z}\cdot x) = (y\cdot(z|x))\cdot a = y\cdot (z|x)a = y\cdot (z|x\cdot a) = \theta_{y,z} \cdot t_a(x)$ for each $x$, so $t_a$ is a $B$-module map.  Notice that $C \to \mc B_B(X); a \mapsto t_a$ is contractive.  This map is also injective, as if $t_a=0$ then $x \cdot a = 0$ for all $x\in E$, so $0 = (y|x\cdot a) = (y|x) a$ for all $x,y$, so $ba=0$ for all $b\in C$ by density, so $a=0$.

A typical member of $E_* \proten_B E$ is $\mu\otimes x$ for $\mu\in E_*$ and $x\in E$, and then for $y,z\in E$, let $a=(y|z)\in C$, so
\[ \ip{ t_{(y|z)} }{ \mu\otimes x } = \ip{ t_a }{ \mu\otimes x }
= \ip{t_a(x)}{\mu} = \ip{x\cdot a}{\mu}
= \ip{x \cdot (y|z)}{\mu} = \ip{\theta_{x,y}\cdot z}{\mu}
= \ip{z}{\mu \cdot \theta_{x,y}}, \]
using that $E_*$ is a right $B$-module.  Using this relation, we see that
\[ \ip{t_{(y|x_0)}}{\mu\otimes x} = \ip{x_0}{\mu\cdot\theta_{x,y}}
= \lim_i \ip{x_i}{\mu\cdot\theta_{x,y}}
= \lim_i \ip{t_{(y|x_i)}}{\mu\otimes x}
= \lim_i \ip{t_{T(e_i\cdot y)^*}}{\mu\otimes x}, \]
in the last step using that $T(e_i\cdot y)^* = (x_i|y)^* = (y|x_i)$.  As $e_i\cdot y \to y$ is norm, also $T(e_i\cdot y)^* \to T(y)^*$ in norm in $B$, so $t_{T(e_i\cdot y)^*} \to t_{T(y)^*}$ in norm, so also weak$^*$.  As elements $\mu\otimes x$ have dense linear span in $E_* \proten_B E$ we conclude that $t_{(y|x_0)} = t_{T(y)^*}$.  Hence $(y|x_0) = T(y)^*$ so $T(y) = (x_0|y)$.  Thus $E$ is self-dual, over $C$.

As $C$ is an ideal in $A$, this also shows that $E$ is self-dual over $A$, see \cite[Lemma~8.5.2]{BlecherLeMerdy_Book} for example.
\end{proof}

We now give an alternative way to show the second part of \cite[Theorem~2.6]{Schweizer_HilbModsPredual}, bypassing the erroneous issue with isomorphic, not isometric, ``preduals'' of $W^*$-algebras.  We collect some useful remarks from (for example) \cite{BMS_quasimults}.  As well as the notion of a right Hilbert $C^*$-module, we also have that of a \emph{left} Hilbert $C^*$-module, say with inner-product (now linear in the first variable) $[\cdot, \cdot]$.  Then a \emph{Hilbert $B$-$A$-bimodule} $E$ is a right Hilbert $C^*$-module, say over $A$, and a left Hilbert $C^*$-module, say over $B$, with the compatibility that $[x,y]\cdot z = x \cdot (y|z)$ for each $x,y,z\in E$.  Given just the right structure, we can always take $B = \mc K(E)$ with the canonical left action, and with $[x,y] = \theta_{x,y}$; indeed, by \cite[Proposition~1.10]{BMS_quasimults} this is essentially the only example.

Given a Hilbert $B$-$A$-bimodule $E$, let $\overline E = \{ \overline x : x\in E \}$ be the conjugate Banach space, made into a $A$-$B$-bimodule for the actions and inner-products
\[ a \cdot \overline x = \overline{x\cdot a^*}, \ 
\overline x\cdot b = \overline{b^*\cdot x}, \quad
[\overline x, \overline y] = (x|y), \ 
(\overline x|\overline y) = [x,y] \qquad (x,y\in E, a\in A, b\in B). \]
One can check that all the requirements are satisfied; for example $[\overline x, \overline y] \cdot \overline z = \overline{z\cdot(y|x)} = \overline{ [z,y]\cdot x } = \overline x \cdot (\overline y|\overline z)$ as $(x|y)^* = (y|x)$ and so forth.  Thus often one can prove results just about $A$, say, and then use this mechanism to translate to an analogous result about $B$.

\begin{theorem}\label{thm:dba_for_mods}
Let $A$ be a $C^*$-algebra and $E$ a Hilbert $C^*$-module over $A$ which is isomorphic (perhaps not isometric) to a dual space.  Suppose the left action of $\mc K(E)$, and the right action of $A$, on $E$ are both weak$^*$-continuous.  With $C = \overline\lin\{(x|y):x,y\in E\}$, we have that $E$ is self-dual over $C$, and that $M(C)$ is a $W^*$-algebra.  Furthermore, treating the inner-product on $E$ as over $M(C)$, the $M(C)$-valued inner-product on $E$ is separately weak$^*$-continuous, and so $E_*$ is isomorphic to the canonical predual constructed by Paschke.
\end{theorem}
\begin{proof}
By Proposition~\ref{prop:KE_wstar_selfdual}, $E$ is self-dual over $A$.  Working with left modules instead (or use $\overline E$, see the discussion above), we see that $E$ is self-dual over $B = \mc K_A(E)$ as well (and note that $E$ is full over $B$).  As observed by Paschke in \cite[Proposition~3.4]{paschke_inner_prod_mods}, when $E$ is self-dual (not necessarily over a $W^*$-algebra) all $B$-module maps are adjointable (the proof is like that for Hilbert spaces).  Hence $\mc B_B(E) = \mc L_B(E)$.

As in the proof of Theorem~\ref{prop:KE_wstar_selfdual}, $\mc B_B(E)$ has (isomorphic) predual $E_* \proten_B E$, which we claim is a dual Banach algebra predual, that is, $E_* \proten_B E$ is a $\mc B_B(E)$-bimodule.  The left action is easy: by linearity and continuity, there is a $\mc B_B(E)$ action on $E_* \proten E$ given by $T\cdot(\mu\otimes x) = \mu\otimes T(x)$.  As $T\in\mc B_B(E)$  commutes with the left action of $B$, this action drops to a well-defined action on the quotient $E^* \proten_B E$.  Then
for $S, T \in\mc B_B(E)$ and $\mu\otimes x \in E_* \proten_B E$ we have $\ip{ST}{\mu\otimes x} = \ip{ST(x)}{\mu} = \ip{S}{T\cdot(\mu\otimes x)}$ as required to show that the $\mc B_B(E)$ action on $E_* \proten_B E$ is compatible with the product on $\mc B_B(E)$.  We now work towards showing the same of the right action.

We claim that $\lin\{ a\cdot \mu : \mu\in E_*, a\in C \}$ is norm dense in $E_*$.  To show this, let $x\in E$ with $0 = \ip{x}{a\cdot\mu} = \ip{x\cdot a}{\mu}$ for each $a\in C, \mu\in E_*$.  Then $x\cdot a=0$ for all $a\in C$, so $x=0$, and the claim follows by Hahn--Banach.  There is then a contractive left action of $M(C)$ on $E_*$ satisfying $\alpha\cdot (a\cdot\mu) = \alpha a \cdot \mu$ for $\alpha\in M(C), a\in C, \mu\in E_*$.  Indeed, for any $\mu' \in \lin\{ a\cdot \mu : \mu\in E_*, a\in C \}$ we have $\alpha \cdot \mu' = \lim_i \alpha e_i \cdot \mu'$ for any approximate identity $(e_i)$ for $C$, from which it follows that the action is well-defined, contractive, and so extends by continuity to all of $E_*$.

By \cite[Proposition~1.10]{BMS_quasimults} the map $t \colon C \to \mc B_B(E)$ considered before, $t_a(x) = x\cdot a$, gives an anti-isomorphism between $C$ and $\mc K_B(E)$, the compact operators on $E$ for the left structure.  Indeed, consider the compact operator $\theta'_{x,y} \colon E\to E; z\mapsto [z,x]\cdot y$.  Then $\theta'_{x,y}(z) = z \cdot (x|y) = t_{(x|y)}(z)$ for each $z$, so $t_{(x|y)} = \theta'_{x,y}$.  This extends to the multiplier algebras to give an anti-isomorphism $t \colon M(C) \to M(\mc K_B(E)) = \mc L_B(E)$ which satisfies $t_\alpha(x\cdot a) = t_\alpha (t_a(x)) = t_{a \alpha}(x) = x\cdot (a \alpha)$ for $x\in E, a\in C, \alpha\in M(C)$, here reversing the product as an anti-homomorphism.  As $\lin\{ x\cdot a : x\in E, a\in C \}$ is dense in $E$ (again, see \cite[page~5]{Lance_HilbModsBook} for example) this relation completely determines $t_\alpha$.

Let $S\in\mc B_B(E) = \mc L_B(E)$, so there is $\alpha\in M(C)$ with $t_\alpha = S$.  For $a\in C, \mu\in E_*$ consider $\mu' = a\cdot \mu \in E_*$ as an element of $E^*$.  As $E$ is a left $\mc B_B(E)$ module, $E^*$ is a right module, and so $\mu'\cdot S$ makes sense.  For $x\in E, b\in C$ let $x' = x\cdot b\in E$, so
\begin{align*}
\ip{\mu'\cdot S}{x'} = \ip{S(x')}{\mu'} = \ip{t_\alpha(x\cdot b)}{a\cdot\mu} = \ip{x\cdot(b \alpha)}{a\cdot\mu}
= \ip{x\cdot(b\alpha a)}{\mu}
= \ip{x'}{(\alpha a)\cdot \mu}.
\end{align*}
By density of such $x'$, we have that $\mu'\cdot S = (\alpha a)\cdot \mu \in E_*$, and so by density of such $\mu'$, we have shown that $E_*$ is a right $\mc B_B(E)$-module.

Then for $S, T \in\mc B_B(E)$ and $\mu\otimes x \in E_* \proten_B E$ we have $\ip{ST}{\mu\otimes x} = \ip{ST(x)}{\mu} = \ip{T(x)}{\mu\cdot S} = \ip{T}{\mu\cdot S \otimes x}$ which shows that $E_* \proten_B E$ is also a right $\mc B_B(E)$-module.  Hence $E_* \proten_B E$ is a predual turning $\mc B_B(E) = \mc L_B(E) \cong M(C)$ into a dual Banach algebra.  By Theorem~\ref{thm:cstar_dba}, $M(C)$ is a $W^*$-algebra, and $E_* \proten_B E$ is isomorphic to $M(C)_*$ (meaning these preduals are equal, when considered as subspaces of $M(C)^*$).

Let $\omega \in M(C)_*$, so there are $(x_n)\subseteq E, (\mu_n)\subseteq E_*$ with $\sum_n \|\mu_n\| \|x_n\| < \infty$ and with $\ip{\alpha}{\omega} = \sum_n \ip{t_\alpha(x_n)}{\mu_n}$ for each $\alpha \in M(C)$.  Given $x,y\in E$, as $(x|y)\in C \subseteq M(C)$, we see that
\[ \omega((y|x))
= \sum_n \ip{t_{(y|x)}(x_n)}{\mu_n} 
= \sum_n \ip{x_n \cdot (y|x)}{\mu_n} 
= \sum_n \ip{\theta_{x_n, y} \cdot x}{\mu_n} 
= \ip{x}{\sum_n \mu_n \cdot \theta_{x_n, y}}, \]
where the final sum converges absolutely in $E_*$, as $E_*$ is a right $\mc K_A(E)$-module.  It follows that $E \to M(C); x \mapsto (y|x)$ is $\sigma(E, E_*)$-$\sigma(M(C), M(C)_*)$ continuous, as claimed.  Similarly $x\mapsto (x|y)$ is weak$^*$-continuous, as the adjoint on $M(C)$ is weak$^*$-continuous.  This continuity condition on the $M(C)$-valued inner-product ensures uniqueness of the predual $E_*$, see \cite[Lemma~8.5.4]{BlecherLeMerdy_Book}, noting that a careful examination of the proof of this lemma shows that the predual is not assumed to be isometric.
\end{proof}

\begin{remark}
An examination of the proof shows that we only need that the action of $C$ on $E$ is weak$^*$-continuous.  However, as $C$ is an ideal in $A$, the argument showing we can extend the action of $C$ on $E_*$ to $M(C)$ also works to show that $E_*$ has a natural $A$ action under this (hence not actually) weaker condition.
\end{remark}

We have an almost immediate application to $W^*$-categories.

\begin{theorem}\label{thm:wstar_from_bimods}
Let $\mc A$ be a $C^*$-catgeory, and for each $A,B\in\mc A$ suppose that $\hom(A,B)$ is isomorphic (not necessarily isometric) to the dual of some Banach space $\hom(A,B)_*$ which is an $A$-$B$-bimodule, compatible under duality with the $B$-$A$-bimodule structure on $\hom(A,B)$.  Then $\mc A$ is a $W^*$-category, $\hom(A,A)_*$ is the predual of the $W^*$-algebra $A$ for each $A\in\mc A$, and $\hom(A,B)_*$ is equal to the canonical predual of $\hom(A,B)$.
\end{theorem}
\begin{proof}
Given $A,B$, again let $C = \overline\lin\{ x^*y : x,y\in \hom(A,B) \}$ and ideal in $A$, and set $D = \overline\lin\{ xy^* : x,y\in\hom(A,B) \}$ an ideal in $B$.  Then $E = \hom(A,B)$ can be considered as a Hilbert $D$-$C$-bimodule, noting that $E$ is full over both $D$ and $C$ (so, for example, $D \cong \mc K_C(E)$, and so forth, by \cite{BMS_quasimults}, compare Section~\ref{sec:mults}).  Theorem~\ref{thm:dba_for_mods} shows that $E$ is self-dual over $C$ and that $M(C)$ is a $W^*$-algebra.  Thus $E$ has a canonical isometric predual given by \cite[Proposition~3.8]{paschke_inner_prod_mods}, this in fact being equal to $\hom(A,B)_*$.  So each $\hom(A,B)$ has an isometric predual and hence $\mc A$ is a $W^*$-category.
\end{proof}

\appendix
\section{Reduction to essential subspace}\label{sec:make_nondeg}

Let $V\subseteq\mc B(K,H)$ be a TRO, and set $C = \overline{V V^\sharp} \subseteq \mc B(H)$ and $D = \overline{V^\sharp V} \subseteq\mc B(K)$, as in Section~\ref{sec:dualTROs}.  The problem we consider here is that $C$ and/or $D$ might not act non-degenerately.  We give a natural way to correct this.

Consider the linking algebra,
\[ L(V) = \begin{pmatrix} C & V \\ V^\sharp & D \end{pmatrix} \subseteq\mc B(H\oplus K). \]

Given a $C^*$-algebra $A\subseteq\mc B(L)$ for some Hilbert space $L$, setting $L_0 = \overline{AL} = \overline\lin\{x\xi : x\in A, \xi\in L\}$ we obtain an invariant subspace of $A$, such that the resulting $*$-homomorphism $A\to\mc B(L_0)$ is non-degenerate and injective (if $xy\xi=0$ for all $y,\xi$ then $xy=0$ for all $y$ so $xx^*=0$ so $x=0$).

The essential subspace of the linking algebra is, by definition,
\begin{align*}
& \overline{\lin}\big\{ (c\xi+u\eta, v^*\xi+d\eta) : c\in C, d\in D, u,v\in V, \xi\in H, \eta\in K \big\}   \\
&= \overline{\lin}\big\{ (xy^*\xi+u\eta, v^*\xi+z^*w\eta) : x,y,z,w,u,v\in V, \xi\in H, \eta\in K \big\}.
\end{align*}
This clearly contains $\overline{\lin} \{ (u\eta, v^*\xi) : u,v\in V, \xi\in H, \eta\in K \} = \overline{V K} \oplus \overline{V^\sharp H}$.  But conversely, $(xy^*\xi+u\eta, v^*\xi+z^*w\eta) = (u\eta, v^*\xi) + (x(y^*\xi), z^*(w\eta))$ which is the sum of two vectors in $\overline{V K} \oplus \overline{V^\sharp H}$.  We conclude that the essential subspace is $\overline{V K} \oplus \overline{V^\sharp H} = H_0 \oplus K_0$, say.

Note that $K_0^\perp = (V^\sharp H)^\perp = \{ \eta : (v^*\xi|\eta)=0 \ (v\in V, \xi\in H) \} = \{ \eta : v\eta=0 \ (v\in V) \}$.  Let $\iota \colon K_0 \to K$ be the inclusion, and $p\colon H\to H_0$ be the projection (the adjoint of the inclusion $H_0\to H$).  It is natural to consider $V_0 = \{ pv\iota : v\in V \} \subseteq \mc B(K_0, H_0)$.  As $V K \subseteq H_0$ we see that really $pv\iota = v|_{K_0}$ for $v\in V$.  By the calculation of $K_0^\perp$, every $v\in V$ is zero on $K_0^\perp$, and so restricting $v$ to $K_0$ doesn't lose any information.  We conclude that the natural map $V \to V_0$ is an isometry.

Consider $V_0 V_0^\sharp$ acting on $H_0$, which has elements of the form $p v \iota \iota^* u^* p^*$ for $u,v\in V$.  As $\iota \iota^*$ is the projection of $K$ onto $K_0\subseteq K$, and $u^* \in V^\sharp$ has image contained in $K_0$, this becomes $p vu^* p^*$.  As $p^*$ is the inclusion $H_0\to H$, and $vu^*$ has image contained in $H_0$, we see that $p vu^* p^* = vu^* |_{H_0}$.  So $C_0 = \overline{ V_0 V_0^\sharp }$ is just $C$ restricted to $H_0$.  Similarly $D_0 = \overline{ V_0^\sharp V_0 }$ is the restriction of $D$ to $K_0$.  We have argued that $V_0$ acting on $K$ is the restriction of $V$ to $K_0$; similarly for $V_0^\sharp$.  In conclusion, the restriction of the linking algebra to $H_0 \oplus K_0$ is exactly the linking algebra for $V_0$.

Notice that $V_0$ is \emph{non-degenerate} in the sense that $V_0K_0$ is dense in $H_0$, and $V_0^\sharp H_0$ is dense in $K_0$, and similarly $C_0$ and $D_0$ are non-degenerate.

Finally, suppose that $V \subseteq \mc B(K,H)$ is a $W^*$-TRO.  As $V_0 \subseteq \mc B(K_0, H_0)$ is just the (co)restriction of $V$, it is easy to see that $V_0$ is also a $W^*$-TRO.

\section{Double duals}\label{sec:biduals}

We wish to show that when $V\subseteq\mc B(H,K)$ is a TRO, also $V^{**}$ is a ($W^*$-) TRO, see Section~\ref{sec:dualTROs}.  Form the linking algebra $L(V)$, and let $L(V)\subseteq\mc B(L)$ be some faithful, non-degenerate representation: we wish to find a decomposition of $L$ such that $L(V)$ becomes a matrix algebra with respect to this decomposition.

Regard $C$ as a subalgebra of $L(V)$ and let $L_C = \overline{CL}$.  Let $(e_i)$ be an approximate identity for $C$, so $e_i \xi \to \xi$ in norm for each $\xi\in L_C$.  Letting $p$ be the orthogonal projection onto $L_C$ we hence have that $e_i p \to p$ strongly.  As $e_i^*(L) \subseteq L_C$ for each $i$, we have that $(1-p)e_i^*=0$ and so $e_i(1-p)=0$.  So for any $\xi\in L$ we have that $e_i\xi = e_ip\xi + e_i(1-p)\xi = e_ip\xi \to p\xi$, and so conclude that $e_i\to p$ strongly.  Similarly set $L_D = \overline{DL}$, let $q$ be the projection onto $L_B$, and let $(f_j)$ be an approximate identity for $B$, so that $f_j \to q$ strongly.

For $u\in V$ we have that $ue_i\to u$ in norm: this follows by computing $(ue_i-u)^*(ue_i-u)$, compare \cite[(1.5)]{Lance_HilbModsBook} for example.  Similarly $f_ju\to u$.  It follows that
\[ e_i + f_j = \begin{pmatrix} e_i & 0 \\ 0 & f_j \end{pmatrix} \in L(V) \]
forms an approximate identity for $L(V)$, where the underlying directed set is given the product order.  As $L(V)$ acts on $L$ non-degenerately, the same argument shows that $e_i+f_j \to 1$ strongly, and so $p+q=1$.  Given a typical element of $L(V)$ and $\xi\in L$ we see that
\begin{align*} p \begin{pmatrix} a & u^* \\ v & b \end{pmatrix} p \xi
&= \lim_i p \begin{pmatrix} a & u^* \\ v & b \end{pmatrix} \begin{pmatrix} e_i & 0 \\ 0 & 0 \end{pmatrix} \xi
= \lim_i p \begin{pmatrix} ae_i & 0 \\ ve_i & 0 \end{pmatrix}\xi
= p \begin{pmatrix} a & 0 \\ v & 0 \end{pmatrix}\xi 
= \lim_i \begin{pmatrix} e_i a & 0 \\ 0 & 0 \end{pmatrix}\xi
= a \xi,
\end{align*}
that is, $L(V) \ni x \mapsto pxp$ is the natural projection of $L(V)$ onto $C$.  Similarly $D = qL(V)q$ and $V = pL(V)q$.  We conclude that with respect to the direct sum $L = L_C \oplus L_D$ we have that $L(V)$ is a matrix algebra (see e.g.\@ \cite{Skeide_matrix_cstar_algs} for more on ``generalised matrix $C^*$-algebras'').

As the projection $L(V)\to C; x\mapsto pxp$ is implemented spatially, it is continuous for the weak$^*$-topology on $\mc B(L)$; similarly the other projections.  It follows that $\overline{L(V)}^\sigma$ is also a matrix algebra, but by von Neumann's theorem, this is $L(V)''$, so
\begin{equation}\label{eq:Mdd}
L(V)'' = \overline{L(V)}^\sigma
= \begin{pmatrix} \overline{C}^\sigma & \overline{V}^{\sigma} \\ \overline{V}^{\sigma \sharp} & \overline{D}^\sigma \end{pmatrix}
= \begin{pmatrix} C'' & \overline{V}^{\sigma} \\ \overline{V}^{\sigma\sharp} & D'' \end{pmatrix},
\end{equation}
where here $C''$ is computed in $\mc B(L_C)$, similarly $D''$, and $\overline{V}^\sigma$ is the weak$^*$ closure of $V$ in $\mc B(L_C, L_D)$.  We have that $(V^\sigma)^\sharp = (V^\sharp)^\sigma$ by weak$^*$-continuity of the adjoint on $\mc B(L)$.

We now choose $L(V) \subseteq \mc B(L)$ to be the universal representation of $L(V)$, e.g.\@ \cite[Definition~III.2.3]{TakesakiI}, so that $L(V)'' = L(V)^{**}$ the bidual, the universal enveloping algebra of $L(V)$.  Given $\mu\in V^*$ we Hahn-Banach extend $\mu$ to $L(V)^*$ with the same norm (or pre-compose with the projection $L(V)\to V$).  Given $\epsilon>0$ there is $\omega\in\mc B(L)_*$ restricting to $\mu$ with $\|\omega\| \leq (1+\epsilon)\|\mu\|$.  Regarding $\omega$ as a trace-class operator on $L$, and $q\omega p$ as trace-class $L_C \to L_D$, we have
\[ \mu(v) = \Tr\bigl( \omega \smattwo{0}{v}{0}{0} \bigr)
= \Tr\bigl( \smattwo{0}{0}{0}{1} \omega \smattwo{1}{0}{0}{0} \smattwo{0}{v}{0}{0} \bigr)
= \Tr(q\omega p v) \qquad (v\in V). \]
This shows that $\mc B(L_D, L_C)_* \to V^*$ is a metric surjection, and so $\overline{V}^\sigma \subseteq \mc B(L_D, L_C)$ is isometric to the bidual $V^{**}$.  We have hence realised $V^{**}$ as a W$^*$-closed TRO.

As an aside, we could now start with $\overline{V}^\sigma = V^{**} \subseteq \mc B(L_D, L_C)$ and consider the linking algebra for this $W^*$-TRO.  This involves considering for example $(\overline{V}^\sigma)^\sharp \overline{V}^\sigma$, which by the form of \eqref{eq:Mdd}, is a subalgebra of $\overline{D}^\sigma = D''$.  As $V^\sharp V$ is norm dense in $D$, $V^\sharp V$ is weak$^*$-dense in $D''$ and so we conclude that $(\overline{V}^\sigma)^\sharp \overline{V}^\sigma = D''$.  Similarly for $C''$.  So the linking algebra of $\overline{V}^\sigma$ is exactly $L(V)''$.

\bibliographystyle{plain}
\bibliography{paper.bib}

\noindent
\texttt{matt.daws@cantab.net} \\
Matthew Daws,  \\
School of Mathematical Sciences \\
Lancaster University \\
Lancaster \\
LA1 4YF \\
United Kingdom

\end{document}